\documentclass[12pt]{amsart}
\usepackage{amsmath,amssymb,amsbsy,amsfonts,latexsym,amsopn,amstext,cite,
                                               amsxtra,euscript,amscd,bm}
\usepackage{url}
\usepackage[colorlinks,linkcolor=blue,anchorcolor=blue,citecolor=blue,backref=page]{hyperref}
\usepackage{color}
\usepackage{graphics,epsfig}
\usepackage{graphicx}
\usepackage{float}

\usepackage{mathtools}
\usepackage{todonotes}

\hypersetup{breaklinks=true}

\RequirePackage{mathrsfs} \let\mathcal\mathscr

\usepackage[norefs,nocites]{refcheck}

\begin{document}
\newtheorem{problem}{Problem}
\newtheorem{theorem}{Theorem}
\newtheorem{lemma}[theorem]{Lemma}
\newtheorem{claim}[theorem]{Claim}
\newtheorem{cor}[theorem]{Corollary}
\newtheorem{prop}[theorem]{Proposition}
\newtheorem{definition}{Definition}
\newtheorem{question}[theorem]{Question}

\numberwithin{equation}{section}
\numberwithin{theorem}{section}

\def\cA{{\mathcal A}}
\def\cB{{\mathcal B}}
\def\cC{{\mathcal C}}
\def\cD{{\mathcal D}}
\def\cE{{\mathcal E}}
\def\cF{{\mathcal F}}
\def\cG{{\mathcal G}}
\def\cH{{\mathcal H}}
\def\cI{{\mathcal I}}
\def\cJ{{\mathcal J}}
\def\cK{{\mathcal K}}
\def\cL{{\mathcal L}}
\def\cM{{\mathcal M}}
\def\cN{{\mathcal N}}
\def\cO{{\mathcal O}}
\def\cP{{\mathcal P}}
\def\cQ{{\mathcal Q}}
\def\cR{{\mathcal R}}
\def\cS{{\mathcal S}}
\def\cT{{\mathcal T}}
\def\cU{{\mathcal U}}
\def\cV{{\mathcal V}}
\def\cW{{\mathcal W}}
\def\cX{{\mathcal X}}
\def\cY{{\mathcal Y}}
\def\cZ{{\mathcal Z}}

\def\A{{\mathbb A}}
\def\B{{\mathbb B}}
\def\C{{\mathbb C}}
\def\D{{\mathbb D}}
\def\E{{\mathbb E}}
\def\F{{\mathbb F}}
\def\G{{\mathbb G}}
\def\I{{\mathbb I}}
\def\J{{\mathbb J}}
\def\K{{\mathbb K}}
\def\L{{\mathbb L}}
\def\M{{\mathbb M}}
\def\N{{\mathbb N}}
\def\O{{\mathbb O}}
\def\P{{\mathbb P}}
\def\Q{{\mathbb Q}}
\def\R{{\mathbb R}}
\def\S{{\mathbb S}}
\def\T{{\mathbb T}}
\def\U{{\mathbb U}}
\def\V{{\mathbb V}}
\def\W{{\mathbb W}}
\def\X{{\mathbb X}}
\def\Y{{\mathbb Y}}
\def\Z{{\mathbb Z}}

\def\ep{{\mathbf{e}}_p}
\def\em{{\mathbf{e}}_m}
\def\eq{{\mathbf{e}}_q}

\def\scr{\scriptstyle}
\def\\{\cr}
\def\({\left(}
\def\){\right)}
\def\[{\left[}
\def\]{\right]}
\def\<{\langle}
\def\>{\rangle}
\def\fl#1{\left\lfloor#1\right\rfloor}
\def\rf#1{\left\lceil#1\right\rceil}
\def\le{\leqslant}
\def\ge{\geqslant}
\def\eps{\varepsilon}
\def\mand{\qquad\mbox{and}\qquad}

\def\sssum{\mathop{\sum\ \sum\ \sum}}
\def\ssum{\mathop{\sum\, \sum}}
\def\ssumw{\mathop{\sum\qquad \sum}}

\def\vec#1{\mathbf{#1}}
\def\inv#1{\overline{#1}}
\def\num#1{\mathrm{num}(#1)}
\def\dist{\mathrm{dist}}

\def\fA{{\mathfrak A}}
\def\fB{{\mathfrak B}}
\def\fC{{\mathfrak C}}
\def\fU{{\mathfrak U}}
\def\fV{{\mathfrak V}}

\newcommand{\bflambda}{{\boldsymbol{\lambda}}}
\newcommand{\bfxi}{{\boldsymbol{\xi}}}
\newcommand{\bfrho}{{\boldsymbol{\rho}}}
\newcommand{\bfnu}{{\boldsymbol{\nu}}}

\def\GL{\mathrm{GL}}
\def\SL{\mathrm{SL}}

\def\Hba{\overline{\cH}_{a,m}}
\def\Hta{\widetilde{\cH}_{a,m}}
\def\Hb1{\overline{\cH}_{m}}
\def\Ht1{\widetilde{\cH}_{m}}

\def\flp#1{{\left\langle#1\right\rangle}_p}
\def\flm#1{{\left\langle#1\right\rangle}_m}
\def\dmod#1#2{\left\|#1\right\|_{#2}}
\def\dmodq#1{\left\|#1\right\|_q}

\def\Zm{\Z/m\Z}

\def\Err{{\mathbf{E}}}

\newcommand{\comm}[1]{\marginpar{%
\vskip-\baselineskip 
\raggedright\footnotesize
\itshape\hrule\smallskip#1\par\smallskip\hrule}}

\def\xxx{\vskip5pt\hrule\vskip5pt}


\title{A refinement of the Burgess bound for character sums}

\author[B. Kerr] {Bryce Kerr}

\address{Department of Pure Mathematics, University of New South Wales,
Sydney, NSW 2052, Australia}
\email{bryce.kerr89@gmail.com}

 \author[I. E. Shparlinski] {Igor E. Shparlinski}

\address{Department of Pure Mathematics, University of New South Wales, 
Sydney, NSW 2052, Australia}
\email{igor.shparlinski@unsw.edu.au}


\author[K. H. Yau]{Kam Hung Yau}

\address{Department of Pure Mathematics, University of New South Wales,
Sydney, NSW 2052, Australia}
\email{kamhung.yau@unsw.edu.au}
 
\date{\today}
\pagenumbering{arabic}


 \begin{abstract} 
In this paper we give a refinement of the  bound of  D.~A.~Burgess for multiplicative character sums modulo a prime number $q$. 
This continues a series of  previous logarithmic improvements, which are mostly due to J.~B.~Friedlander,  H.~Iwaniec and E.~Kowalski. In particular, 
for any nontrivial multiplicative character $\chi$ modulo a prime $q$ and any integer $r\ge 2$, we show that 
$$ \sum_{M<n\le M+N}\chi(n) = O\( N^{1-1/r}q^{(r+1)/4r^2}(\log q)^{1/4r}\), 
$$
which sharpens previous results   by a factor $(\log q)^{1/4r}$.  
 Our improvement comes from  averaging over numbers with no small prime factors rather than over an interval
 as in previous approaches. 
 \end{abstract}

\maketitle
\section{Introduction}
Given a prime number $q$ and a multiplicative character $\chi$ modulo $q$, we consider bounding the sums
\begin{equation}
\label{eq:charsums}
\sum_{M<n\le M+N}\chi(n).
\end{equation}
 The first nontrivial result in this direction, which is   about a century old,  is due to  P{\'o}lya~\cite{Pol} 
 and Vinogradov~\cite{Vin}  and takes the form
\begin{equation}
\label{eq:polyavinogradov}
\sum_{M<n\le M+N}\chi(n) = O\(q^{1/2}\log{q}\) 
\end{equation}
with an absolute implied constant. 
Clearly the   bound~\eqref{eq:polyavinogradov} is nontrivial provided $N\ge q^{1/2} (\log q)^{1+\varepsilon}$ for any fixed $\varepsilon >0$. 

Several  logarithmic improvements of~\eqref{eq:polyavinogradov} have recently been obtained for special characters, see~\cite{FrGo,GoLa,LaMa}
and references therein.

For large values of $N,$ the Polya-Vinogradov bound~\eqref{eq:polyavinogradov}  is still the 
sharpest result known today although 
Montgomery and Vaughan~\cite{MV} have shown that assuming the truth of the Generalized Riemann Hypothesis we have
$$
\sum_{M<n\le M+N} \chi(n) = O\(q^{1/2}\log\log{q}\).
$$
The P{\'o}lya--Vinogradov bound~\eqref{eq:polyavinogradov} can be thought of as roughly saying that for large $N$, the sequence $\{\chi(n)\}_{n=M+1}^{M+N}$ behaves like a typical random sequence chosen uniformly from the image $\chi(\{1,\dots,q-1\})$. We  expect this to be true for smaller values of $N$ although this problem is much less understood. In the special case $M=0$, the Generalized Riemann Hypothesis (GRH) implies that
\begin{equation}
\label{eq:GRH}
\left|\sum_{0<n\le N}\chi(n)\right|\le N^{1/2}q^{o(1)},
\end{equation}
which is nontrivial provided $N\ge q^{\varepsilon}$ and is essentially optimal. Although the conditional bound~\eqref{eq:GRH} on the GRH is well-known, see for example~\cite[Section~1]{MV}; 
 it may not be easy to find a direct reference, however it can be easily derived from~\cite[Theorem~2]{GrSo1}.

We also note that Tao~\cite{Tao} has shown progress on the generalized Elliott-Halberstam conjecture allows one to bound short character sums in the case $M=0$.

For values of $N$ below the  P{\'o}lya--Vinogradov range, the sharpest unconditional bound for the sums~\eqref{eq:charsums} is due to Burgess~\cite{Bur1,Bur2}  and may be stated as follows. For any prime number $q$,  nontrivial multiplicative character $\chi$ modulo $q$ and integer $r\ge 1$ we have
\begin{equation}
\label{eq:Burg}
\sum_{M<n\le M+N}\chi(n) = O\( N^{1-1/r}q^{(r+1)/4r^2}\log q\),
\end{equation}
where the implied constant may depend on $r$, 
and is nontrivial provided $N\ge q^{1/4+\varepsilon}$ for any fixed $\varepsilon >0$.  This bound has remained the sharpest for short sums over the past fifty years although slight refinements have been
 made by improving the factor $\log{q}$. 
 For example, by~\cite[Equation~(12.58)]{IwKo} we have 
\begin{equation}
\label{eq:Iw1}
\sum_{M<n\le M+N}\chi(n) = O\(N^{1-1/r}q^{(r+1)/4r^2}(\log q)^{1/r}\),
\end{equation}
where the  implied constant  is absolute.
It is also announced in~\cite[Chapter~12, Remark, p.~329]{IwKo}, that one can actually obtain 
\begin{equation}
\label{eq:Iw2}
\sum_{M<n\le M+N}\chi(n) = O\(N^{1-1/r}q^{(r+1)/4r^2}(\log q)^{1/2r}\),
\end{equation}
provided $r\ge 2$, see also~\cite[Theorem~9.27]{MV}. 

We also remark that in the initial range, that is, for $M=0$, slight improvements of the bounds~\eqref{eq:Iw1}, \eqref{eq:Iw2} 
and also of  Theorem~\ref{thm:main1} below are given in~\cite{Ell,GrSo1,GrSo2,Hild}.
However these improvements do not imply any improvement of the above bound of Burgess~\cite{Bur0} on the smallest quadratic nonresidue, we also refer  
to~\cite{BobGold} for a discussion.  Finally, we recall that the best known bounds on the smallest quadratic nonresidue is $O(q^{1/4\sqrt{e}+o(1)})$, and on the gaps between quadratic nonresidues is $O(q^{1/4} \log q)$, are both due to Burgess~\cite{Bur0,Bur3}.  

We recall that both~\eqref{eq:Iw1} and~\eqref{eq:Iw2}  are based on a bound 
of Friedlander and  Iwaniec~\cite[Section~4]{FrIw}  on the number of solutions 
to the congruence~\eqref{eq:multcongeq} however with variables $u_1,u_2$ 
from the whole interval, without any arithmetic constraints. Imposing such constraints
is a new idea which underlines our approach.  Using this idea,  we give a further refinement of the Burgess bound~\eqref{eq:Burg} and thus contribute to the series of logarithmic improvements~\eqref{eq:Iw1} and~\eqref{eq:Iw2}. More specifically, we  improve~\eqref{eq:Iw2} by a factor $(\log{q})^{1/4r}$. We remark that Booker~\cite{Book} has previously used shifts by products $u_1 u_2$ 
where one of the variables is prime in the Friedlander-Iwaniec approach~\cite{FrIw} in order to obtain a numerically explicit Burgess bound. The benefit of prime shifts being simpler computations with estimating greatest common divisors.   
 Our argument can be considered as an elaboration of this idea and our improvement comes from averaging over numbers with no small prime factors rather than over an entire interval which we give in Section~\ref{sec:cong}.

\section{Main result}

Throughout the paper, the implied constants in 
the symbols `$O$' and `$\ll$' may occasionally, where obvious, depend on 
 the real parameter $A$ and 
are absolute otherwise (we recall that $U \ll V$  and $U = O(V)$  are  equivalent
to $|U|\le c V$ for some constant $c$).

Our  main result is as follows.

\begin{theorem}
\label{thm:main1}
Let $q$ be prime, $r\ge 2$, $M$ and $N$ integers with 
$$N\le q^{1/2+1/4r}.$$
For any nontrivial multiplicative character $\chi$  modulo $q$, we have
$$
\sum_{M<n\le M+N}\chi(n) \ll N^{1-1/r}q^{(r+1)/4r^2} (\log q)^{1/4r},
$$
where the implied constant
is absolute. 
\end{theorem}

\section{Preliminary results}

First we recall the following bound  which is contained in~\cite[Theorem~1.2]{Trev} which is well-known (with slightly weaker constants).  
%
\begin{lemma}
\label{lem:weil}
Let $q$ be prime and $\chi$ a nontrivial multiplicative character  modulo $q$. Then we have
$$\sum_{\lambda=1}^{q}\left|\sum_{1\le v \le V}\chi(\lambda+v) \right|^{2r}\le (2r)^{r}V^{r}q+2rV^{2r}q^{1/2}.$$
\end{lemma}

For any positive real numbers $w$ and $z$ we denote
$$
V(w) = \prod_{p < w}\left(1-\frac{1}{p}\right),
$$
and
\begin{equation}
\label{eq:Pz}
P(z) = \prod_{p <z} p.
\end{equation}
It follows from  Mertens formula (see~\cite[Equation~(2.16)]{IwKo}) that
\begin{equation}
\label{eq:Mert}
\frac{1}{\log{w}}\ll V(w) \ll \frac{1}{\log w}. 
\end{equation}

As usual, we use $(u,v)$ to denote the greatest common divisor of two integers $u$ and $v$. 

For real $U$ and $z$, we define the set $\cU_z(U)$ by 
\begin{equation}
\label{eq:Uz}
\cU_z(U)=\{ 1\le u \le U ~:~(u,P(z))=1  \},
\end{equation}
where $(a,b)$ denotes the greatest common divisor of integers $a$ and $b$.

The following result follows from combining~\cite[Theorem~4.1]{DHG} with arguments from the proof of~\cite[Lemma~4.3]{DHG}. We also refer the reader to~\cite[Equation~(6.104)]{FI}.

\begin{lemma} \label{lem:Ucard} Let $C$ be sufficiently large and suppose that  
\begin{equation}
\label{eq:zUin}
z^C \le U.
\end{equation}
 Then for the cardinality of $\cU_z(U)$,  we have 
$$
\frac{U}{\log z}\ll |\cU_z(U)|\ll \frac{U}{\log z}.
$$
\end{lemma}

\begin{proof}
Let 
$$\cA=\{1,\dots,U\},$$
so that with notation as in~\cite[Theorem~4.1]{DHG} we have
$$|\cU_z(U)|=S(\cA,\cP,z),$$
and hence by~\cite[Theorem~4.1]{DHG}, for any $v\ge 1$ we have 
\begin{align*}
|\cU_z(U)|=UV(z)&\left(1+O\left(\exp(-v\log{v}-3v/2\right) \right)\\
& \qquad \qquad \qquad +O\left(\sum_{\substack{n<z^{2v} \\ n|P(z)}}3^{\nu(n)}|r_\cA(n)| \right).
\end{align*}
Considering the last term on the right
$$
\sum_{\substack{n<z^{2v} \\ n|P(z)}}3^{\nu(n)}|r_\cA(n)|\ll \sum_{\substack{n<z^{2v} \\ n|P(z)}}3^{\nu(n)}\le z^{2v}\sum_{\substack{ n|P(z)}}\frac{3^{\nu(n)}}{n}\le z^{2v}\prod_{p<z}\left(1+\frac{3}{p} \right),
$$
and since
$$\prod_{p<z}\left(1+\frac{3}{p} \right)\le \prod_{p<z}\left(1-\frac{1}{p} \right)^{-3}=V(z)^{-3},$$
we obtain
\begin{align*}
\sum_{\substack{n<z^{2v} \\ n|P(z)}}3^{\nu(n)}|r_\cA(n)|\ll z^{2v}V(z)^{-3},
\end{align*}
which implies that 
\begin{equation}
\label{eq:Uin1}
|\cU_z(U)|=UV(z)\left(1+ \mathrm{ET}_1 \right) +  \mathrm{ET}_2.
 \end{equation}
 with the error terms 
$$
 \mathrm{ET}_1 = O\(\exp\(-v\log{v}-3v/2\)\) \mand 
  \mathrm{ET}_2 = O\(z^{2v}V(z)^{-3}\).
$$
Let $\varepsilon$ be sufficiently small and take
$$v = \frac{(1-\varepsilon)}{2}\frac{ \log U}{\log z}.$$
Then we have  
\begin{equation}
\label{eq:Uin2}
  \mathrm{ET}_2 \ll  U^{1-\varepsilon}(\log z)^{3},
\end{equation}
and by~\eqref{eq:zUin} we may choose $C$ such that   \begin{equation}
\label{eq:Uin3}
  \mathrm{ET}_1 \le \frac{1}{2}.
\end{equation}

Combining~\eqref{eq:Uin1} with~\eqref{eq:Uin2} and~\eqref{eq:Uin3}, we derive 
$$UV(z)\ll |\cU_z(U)|\ll UV(z),$$
and the result follows from the Mertens estimate~\eqref{eq:Mert}.
\end{proof}

We recall a simplified form of~\cite[Lemma~4.4]{DHG}.

\begin{lemma} \label{lem:sieveboud}
For any integers $t$, $z,$   any real $U \ge 1$ and any positive constant $0 < A <1/2$, we have
$$
\sum_{\substack{u \in \cU_z(U) \\ t\mid  u}} 1 \ll  Ut^{-1}V(z),
$$
if $z < (Ut^{-1})^{A}$ and
$$
\sum_{\substack{u \in \cU_z(U) \\ t\mid  u}} 1  \ll  Ut^{-1}V(Ut^{-1}), 
$$
if $(Ut^{-1})^{A} \le z$.
\end{lemma}

Note that Lemma~\ref{lem:sieveboud} is nontrivial only if $(t,P(z))=1$.

\section{Congruences with numbers with no small prime divisors}
\label{sec:cong}

The  new ingredient underlying our argument is the following:

\begin{lemma}
\label{lem:multcong}
Let $q$ be  prime and $z,M$, $N$ and $U$ integers with 
$$
U\le N, \quad UN\le q.
$$
Fix a sufficiently small positive real number $0 < A <1/2$ and suppose $z$ satisfies
\begin{equation}
\label{eq:zUcond}
1 < z \le U^{A}.
\end{equation}
Let $P(z)$ and   $\cU_z(U)$  be given by~\eqref{eq:Pz}  and by~\eqref{eq:Uz}, respectively, 
and let $I(z,M,N,U)$ count the number of solutions to the congruence
\begin{equation}
\label{eq:multcongeq}
n_1u_1\equiv n_2u_2 \pmod{q},
\end{equation}
with integral variables satisfying
$$
M<n_1,n_2\le M+N \mand  u_1,u_2 \in \cU_z(U).
$$
Then we have 
$$
I(z,M,N,U) \ll  N|\cU_z(U)| \(1 + \frac{\log U}{(\log  z)^2} \).
$$

\end{lemma}
\begin{proof}
For each pair of integers $u_1$ and $u_2$, we let $J(u_1,u_2)$ count the number of solutions to the congruence~\eqref{eq:multcongeq} in variables $n_1,n_2$ satisfying
$$M<n_1,n_2\le M+N,$$
so that
\begin{align*}
I(z,M,N,U)&=\sum_{\substack{u_1,u_2 \in \cU_z(U)}}J(u_1,u_2) \\
                &=\sum_{\substack{u_1 \in \cU_z(U)}}J(u_1,u_1)+2\sum_{ \substack{u_1,u_2 \in \cU_z(U) \\ u_1 < u_2 }}J(u_1,u_2).
\end{align*}
Since 
$$J(u_1,u_1)=N,$$
we have 
$$
I(z,M,N,U)=N\sum_{\substack{u_1 \in \cU_z(U)}}1+
2\sum_{ \substack{u_1,u_2 \in \cU_z(U) \\ u_1 < u_2  }}J(u_1,u_2).
$$
Using  Lemma~\ref{lem:sieveboud} (with $t=1$), the bound~\eqref{eq:Mert} and recalling~\eqref{eq:zUcond} we see that  
$$
\sum_{\substack{u_1 \in \cU_z(U) }}1 \ll \frac{U}{\log z}, 
$$ 
and hence 
\begin{equation}
\label{eq:s1}
I(z,M,N,U) \ll \frac{NU}{\log z}+\sum_{ \substack{u_1,u_2 \in \cU_z(U) \\ u_1 < u_2 }}J(u_1,u_2).
\end{equation}

Fix some pair $u_1,u_2$ with $u_1<u_2$ and consider $J(u_1,u_2)$. We first note that $J(u_1,u_2)$ is bounded by the number of solutions to the equation
\begin{equation}
\label{eq:red1}
u_1(M+n_1)-u_2(M+n_2)=kq,
\end{equation}
with variables $n_1,n_2,k$ satisfying 
$$1\le n_1,n_2\le N, \quad k\in \Z.$$
Since
$$|kq-(u_1-u_2)M|\le UN<q,$$
there exists at most one value $k$ satisfying~\eqref{eq:red1} and hence $J(u_1,u_2)$ is bounded by the number of solutions to the equation~\eqref{eq:red1} with variables satisfying
$$1\le n_1,n_2 \le N.$$
Since we may suppose $J(u_1,u_2)\ge 1$, fixing one solution $n_1^{*},n_2^{*}$ to~\eqref{eq:red1}, for any other solution $n_1,n_2$ we have 
$$u_1(n_1-n_1^{*})=u_2(n_2-n_2^{*}).$$
The above equation determines the residue of $n_1$ modulo $u_2/(u_1,u_2)$ and for each value of $n_1$ there exists at most one solution $n_2$. Since $U\le N$ this implies that 
$$J(u_1,u_2)\ll N\frac{(u_1,u_2)}{u_2},$$
and hence by~\eqref{eq:s1},  we derive  
\begin{equation}
\label{eq:sums1}
 I(z,M,N,U) \ll \frac{NU}{\log z} +N\sum_{u_1,u_2 \in \cU_z(U)}\frac{(u_1,u_2)}{u_2} .
 \end{equation}

Considering the last sum on the right hand side and collecting together $u_1$ and $u_2$ with the same value 
$(u_1,u_2) = d$,  we have 
\begin{equation}
\begin{split}
\label{eq:sums2}
\sum_{u_1,u_2 \in \cU_z(U)}\frac{(u_1,u_2)}{u_2} 
& \le \sum_{d \in \cU_z(U)}d\sum_{\substack{ u_2\in  \cU_z(U) \\ d \mid u_2}}
 \frac{1}{u_2}\sum_{\substack{ u_1 \in \cU_z(u_2)\\ d \mid u_1}}1 \\
 & =   \Sigma_{1} + \Sigma_{2},
\end{split}
\end{equation}
where
$$
\Sigma_{1} = \sum_{d \in \cU_z(U)}d\sum_{\substack{ u_2\in  \cU_z(U) \\ d \mid u_2 \\ z \le (u_{2}/d)^{A} }}
 \frac{1}{u_2}\sum_{\substack{ u_1 \in \cU_z(u_2)\\ d \mid u_1}}1,
$$
and
$$
\Sigma_{2} = \sum_{d \in \cU_z(U)}d\sum_{\substack{ u_2\in  \cU_z(U) \\ d \mid u_2 \\ z > (u_{2}/d)^{A} }}
 \frac{1}{u_2}\sum_{\substack{ u_1 \in \cU_z(u_2)\\ d \mid u_1}}1. 
$$

Considering $\Sigma_1$, by Lemma~\ref{lem:sieveboud} and the condition $z \le (u_{2}/d)^{A}$ we bound
\begin{align*}
\Sigma_{1} & \ll  \sum_{d \in \cU_z(U)} \sum_{\substack{u_2\in  \cU_z(U) \\ d \mid u_2 \\ z \le  (u_{2}/d)^{A} }} V(z) \ll  V(z) \sum_{\substack{d \in \cU_z(U) }} \sum_{\substack{ u_2\in  \cU_z(U) \\ d \mid u_2 \\ z \le  (u_{2}/d)^{A} }}  1.
\end{align*}
The condition $z \le (u_2/d)^{A}$ in the innermost summation implies that  the outer summation over $d$ is non empty only if $z \le (U/d)^{A}$ and hence by  Lemma~\ref{lem:sieveboud} we have
\begin{equation} \label{eq:Sigma111}
\Sigma_{1} \ll V(z) \sum_{\substack{d \in \cU_z(U) \\ z \le  (U/d)^{A} }} \sum_{\substack{u_2\in  \cU_z(U) \\ d \mid u_2 }}  1 \ll  U V(z)^{2}\sum_{\substack{d \in \cU_z(U)  }} d^{-1}.
\end{equation} 
Let
$$
S(t)=\sum_{d \in \cU_z(t)}  1.
$$
Hence applying partial summation and Lemma~\ref{lem:sieveboud}, we obtain
\begin{equation}
\label{eq:sum d}
\begin{split}
\sum_{d \in \cU_z(U)} d^{-1}&  = \frac{S(U)}{U}+ \int_{1}^{U} \frac{S(t)}{t^2}dt \\
& \ll V(z) + \int_{1}^{z^{1/A}} \frac{S(t)}{t^2}dt +\int_{z^{1/A} }^{U}\frac{S(t)}{t^2}dt.
 \end{split}
\end{equation}
For the first integral, bounding trivially $S(t) \le t$, we derive
 \begin{equation}
 \label{eq:int1}
\begin{split}
 \int_{1}^{z^{1/A}} \frac{S(t)}{t^2}dt  \ll \int_{1}^{z^{1/A}} \frac{1}{t} dt \ll \log z.
\end{split}
\end{equation}
For the second integral, after applying Lemma~\ref{lem:sieveboud},  we  have
\begin{equation}
\label{eq:int2}
\int_{z^{1/A} }^{U}\frac{S(t)}{t^2}dt \ll  V(z)\int_{1}^{U}\frac{1}{x}dx \ll V(z) \log U.
\end{equation}
Substituting~\eqref{eq:int1} and~\eqref{eq:int2} in~\eqref{eq:sum d}    we obtain
 $$
 \sum_{d \in \cU_z(U)} d^{-1} \ll V(z) + \log z + V(z)\log U. 
$$
In turn, substituting this inequality in~\eqref{eq:Sigma111}  and recalling  the Mertens
 estimate~\eqref{eq:Mert} on $V(z)$, we derive
\begin{equation}
\label{eq:sigma1b}
\begin{split}
\Sigma_{1} & \ll UV(z)^{2}\(V(z) + \log  z +\frac{\log U}{\log z}\) \\
 & \ll \frac{U}{\log z}\left(1+ \frac{\log U}{(\log  z)^2} \right).
 \end{split}
\end{equation}
It  remains to bound $\Sigma_{2}$. Note that 
\begin{align*}
\Sigma_{2}& = \sum_{d \in \cU_z(U)}d\sum_{\substack{u_2 \in \cU_z\(\min \{U, z^{1/A}d \}\)\\ d \mid u_2   }}
 \frac{1}{u_2}\sum_{\substack{ u_1 \in \cU_z(u_2)\\ d \mid u_1}}1  \\
 & =\Sigma_{21} +\Sigma_{22},
\end{align*}
where
$$
\Sigma_{21} =\sum_{d \in \cU_z\(Uz^{-1/A}\)}d\sum_{\substack{ u_2 \in  \cU_z\(z^{1/A} d\)\\ d \mid u_2 }}
 \frac{1}{u_2}\sum_{\substack{ u_1 \in \cU_z(u_2)\\ d \mid u_1}}1,
$$
and
$$
\Sigma_{22} = \sum_{\substack{d \in \cU_z(U) \\ d > U z^{-1/A}}}
d\sum_{\substack{u_2 \in  \cU_z(U) \\ d \mid u_2 }}
 \frac{1}{u_2}\sum_{\substack{ u_1 \in \cU_z(u_2)\\ d \mid u_1}}1.
$$

Bounding the innermost sum of $\Sigma_{21}$ trivially, we have
$$
\Sigma_{21}   \ll  \sum_{d \in \cU_z\(Uz^{-1/A}\)}   \sum_{\substack{u_2 \in \cU_z\(z^{1/A} d\) \\ d \mid u_2 }} 1.
$$
Noting that $z \le (z^{1/A})^{A}$, an application of Lemma~\ref{lem:sieveboud} gives
\begin{equation}
\label{eq:sigma21b}
\Sigma_{21} \ll z^{1/A}V(z) \sum_{\substack{  d \in \cU_z(Uz^{-1/A})}} 1\ll V(z)U
\ll \frac{U}{\log z}.
\end{equation}
It remains to bound $\Sigma_{22}$. Recalling that 
$$
\Sigma_{22}  = \sum_{\substack{d \in \cU_z(U) \\ d > U z^{-1/A}}}d\sum_{\substack{u_2 \in \cU_z(U) \\ d \mid u_2  }}
 \frac{1}{u_2}\sum_{\substack{ u_1 \in \cU_z(u_2)\\ d \mid u_1}}1,
$$
by Lemma~\ref{lem:sieveboud} and noting that $z > (u_2/d)^{A}$ since $d >Uz^{-1/A} $, we obtain
$$
 \Sigma_{22}   \ll \sum_{\substack{d \in \cU_z(U) \\ d > U z^{-1/A}}}\sum_{\substack{ u_2 \in \cU_z(U)\\ d \mid u_2}
 } V\left(\frac{u_{2}}{d}\right). 
$$
 Let 
\begin{equation}
\label{eq:Rd}
R_{d} =\frac{ \log (U/d)}{\log z}.
\end{equation}
Then $R_{d} \ge 1$ if and only if $d \le Uz^{-1}$ and hence
$$
\Sigma_{22}   \ll \sum_{\substack{d \in \cU_z(U/z) \\ d > U z^{-1/A}}}  \sum_{1 \le r \le R_d} 
 \sum_{\substack{ u_2 \in \cU_z(dz^{r}) \\ d \mid u_2 \\  u_2 \ge  dz^{r-1}}} V\left(\frac{u_{2}}{d}\right).
$$
Fixing a value of $r$ and considering the innermost summation over $u_2$, since $z^{r-1} \le u_{2}/d$ we have $V(u_2/d) \le V(z^{r-1})$ and therefore we assert
$$
\Sigma_{22}   \ll \sum_{\substack{d \in \cU_z(U/z) \\ d > U z^{-1/A}}}  \sum_{1 \le r \le R_d }  V(z^{r-1})  \sum_{\substack{u_2 \in \cU_z(dz^{r})\\ d \mid u_2 }} 1.
$$

Appealing to Lemma~\ref{lem:sieveboud} and separating the term $r=1$, we have
\begin{align*}
\Sigma_{22} & \ll \sum_{\substack{d \in \cU_z(U/z) \\ d > U z^{-1/A}}}  \sum_{1 \le r \le R_d}  V(z^{r-1}) z^{r} \max \{ V(z), V(z^{r}) \} \\
& \ll V(z) \sum_{\substack{d \in \cU_z(U/z) \\ d > U z^{-1/A}}} z + V(z)^2\sum_{\substack{d \in \cU_z(U/z) \\ d > U z^{-1/A}}}  \sum_{2 \le r \le R_d}z^{r},
\end{align*}
so that bounding the first sum trivially gives
$$
\Sigma_{22} \ll UV(z) + V(z)^2\sum_{\substack{d \in \cU_z(U/z) \\ d > U z^{-1/A}}}  \sum_{2 \le r \le R_d}z^{r}.
$$
We see from~\eqref{eq:Rd} that $z^{R_{d}} = Ud^{-1}$ and hence
$$\sum_{2 \le r \le R_d}z^{r}\ll z^{R_d}=\frac{U}{d},$$
which implies that
$$
 \Sigma_{22}   \ll UV(z)+ UV(z)^2\sum_{\substack{d \in \cU_z(U/z) \\ d > U z^{-1/A}}}\frac{1}{d},
$$
and hence
\begin{equation}
\begin{split}
\label{eq:sigma22b}
\Sigma_{22} &\ll UV(z)^{2} + U V(z)^2\sum_{\substack{U z^{-1/A}<  d \le U}}\frac{1}{d}\\
& \ll U V(z)^2  \log z\ll \frac{U}{\log z}.
\end{split}
\end{equation}

Combining~\eqref{eq:sums2},~\eqref{eq:sigma1b},~\eqref{eq:sigma21b} and~\eqref{eq:sigma22b} we get
$$
\sum_{u_1,u_2 \in \cU_z(U)}\frac{(u_1,u_2)}{u_2} \ll \frac{U \log U}{(\log  z)^3}+\frac{U}{\log z},
$$
and hence by~\eqref{eq:sums1}
$$
I(z,M,N,U) \ll \frac{NU}{\log z} \(1 + \frac{\log U}{(\log  z)^2} \),
$$
which together with Lemma~\ref{lem:Ucard} completes the proof since $A$ is assumed sufficiently small.
\end{proof}

\section{Proof of Theorem~\ref{thm:main1}}

We fix an integer $r\ge 2$ and  proceed by induction on $N$. We formulate our induction hypothesis as follows. 
There exists some constant $c_1$, to be determined later,  such that for any integer $M$ and any integer $K<N$ we have
$$
\left|\sum_{M<n\le M+K}\chi(n)\right|\le c_1 K^{1-1/r}q^{(r+1)/4r^2} (\log q)^{1/4r},
$$
and we aim to show that 
\begin{equation}
\label{eq:wts}
\left|\sum_{M<n\le M+N}\chi(n)\right|\le c_1 N^{1-1/r}q^{(r+1)/4r^2}  (\log q)^{1/4r}
\end{equation} 
with an absolute constant $c_1$.  
Since the result is trivial for $N<q^{1/4}$ this forms the basis of our induction. We define the integers $U$ and $V$ by
\begin{equation}
\label{eq:UVdef}
U=\left \lfloor \frac{N}{16 r q^{1/2r}} \right \rfloor \mand V=\left \lfloor  r  q^{1/2r} \right \rfloor,
\end{equation}
and note that 
\begin{equation}
\label{eq:UVbound}
UV\le \frac{N}{16}.
\end{equation} 
We also note that with this choice of $V$ the bound of Lemma~\ref{lem:weil} becomes 
\begin{equation}
\label{eq:weil1}
\sum_{\lambda=1}^{q}\left|\sum_{1\le v \le V}\chi(\lambda+v) \right|^{2r}\le (2r)^{r}V^{r}q+2rV^{2r}q^{1/2}
\le (2r)^{2r} q^{3/2}.
\end{equation}

For any integers $1\le u\le U$ and $1\le v\le V$ we have 
\begin{align*}
\sum_{M<n\le M+N}\chi(n)&=\sum_{M-uv<n\le M+N-uv}\chi(n+uv)
\\ & =\sum_{M<n\le M+N}\chi(n+uv) \\ & \quad +\sum_{M-uv<n\le M}\chi(n+uv)-\sum_{M+N-uv<n\le M+N}\chi(n+uv).
\end{align*}
By~\eqref{eq:UVbound} and our induction hypothesis we have 
$$
\left|\sum_{M-uv<n\le M}\chi(n+uv)\right|\le \frac{c_1}{4} N^{1-1/r}q^{(r+1)/4r^2}  (\log q)^{1/4r},
$$
and 
$$
\left|\sum_{M+N-uv<n\le M+N}\chi(n+uv)\right|\le \frac{c_1}{4} N^{1-1/r}q^{(r+1)/4r^2}  (\log q)^{1/4r},
$$
which combined with the above implies that 
\begin{align*}
\left|\sum_{M<n\le M+N}\chi(n)-\sum_{M<n\le M+N}\chi(n+uv) \right|&\\
 \le \frac{c_1}{2} &N^{1-1/r}q^{(r+1)/4r^2} (\log q)^{1/4r}.
\end{align*}
Let 
\begin{equation}
\label{eq:zdef}
z=\exp\left((\log U)^{1/2} \right),
\end{equation}
and let $P(z)$ and  $\cU_z(U)$ be defined by~\eqref{eq:Pz} and~\eqref{eq:Uz}, respectively.

Averaging over $u\in \cU_z(U)$ and $1\le v \le V$ we see that 
\begin{equation}
\label{eq:Win}
\left|\sum_{M<n\le M+N}\chi(n)\right|\le \frac{1}{|\cU_z(U)|V}W+\frac{c_1}{2} N^{1-1/r}q^{(r+1)/4r^2} (\log q)^{1/4r},
\end{equation}
where
\begin{equation}
\label{eq:WW123}
W=\sum_{M<n\le M+N}\sum_{u\in \cU_z(U)}\left|\sum_{1\le v \le V}\chi(n+uv) \right|.
\end{equation}

By multiplying the innermost summation in~\eqref{eq:WW123} by $\chi(u^{-1})$ and collecting the values of $nu^{-1}  \pmod q$, we  arrive at
$$
W =\sum_{\lambda=1}^{q}I(\lambda)\left|\sum_{1\le v \le V}\chi(\lambda+v) \right|,
$$
where $I(\lambda)$ counts the number of solutions to the congruence
$$
n \equiv  \lambda u  \pmod{q}, \qquad 
M<n \le M+N, \  u \in \cU_z(U).
$$
Writing 
$$
W =\sum_{\lambda=1}^{q} I(\lambda)^{(r-1)/r}  (I(\lambda)^2)^{1/2r} \left|\sum_{1\le v \le V}\chi(\lambda+v) \right|,
$$
we see that  the H\"{o}lder inequality gives
$$
W^{2r}\le \left(\sum_{\lambda=1}^{q}I(\lambda) \right)^{2r-2}\left(\sum_{\lambda=1}^{q}I(\lambda)^2 \right)\left(\sum_{\lambda=1}^{q}\left|\sum_{1\le v \le V}\chi(\lambda+v) \right|^{2r}\right).
$$
From Lemma~\ref{lem:Ucard} we have 
\begin{equation}
\label{eq:lin1}
\sum_{\lambda=1}^{q}I(\lambda)=\sum_{M<n\le M+N}\sum_{u\in \cU_z(U)}1=N|\cU_z(U)|\ll \frac{NU}{\log z}.
\end{equation}
We have 
$$\sum_{\lambda=1}^{q}I(\lambda)^2= I(z,M,N,U),$$
where $I(z,M,N,U)$ is as in  Lemma~\ref{lem:multcong}. 
Since $N<q^{1/2+1/4r}$ the conditions of Lemma~\ref{lem:multcong} are satisfied, hence recalling~\eqref{eq:zdef} we have

\begin{equation}
\label{eq:mult1}
\sum_{\lambda=1}^{q}I(\lambda)^2\ll \frac{NU}{\log z}\left(1+\frac{\log{U}}{(\log  z)^2} \right)\ll \frac{NU}{\log z}.
\end{equation}
Combining and~\eqref{eq:weil1}, \eqref{eq:lin1}, and~\eqref{eq:mult1} gives
$$
W^{2r}\ll (2r)^{2r} q^{3/2} \left(\frac{NU}{\log z}\right)^{2r-1},
$$ 
which using Lemma~\ref{lem:Ucard} we rewrite as 
$$
\(\frac{W}{|\cU_z(U)| }\)^{2r}\ll (2r)^{2r} q^{3/2}N^{2r-1} \frac{\log z}{U}. 
$$
Hence  
\begin{equation}
\label{eq:W prelim}
\frac{W}{|\cU_z(U)|V}\ll r \frac{(\log z)^{1/2r}}{VU^{1/2r}}N^{1-1/2r}q^{3/4r}. 
\end{equation}
Using that $r^{1/r} \le 2$ and recalling~\eqref{eq:UVdef} and~\eqref{eq:zdef}, we see that 
$$V/r \gg q^{1/2r}, \qquad  U^{1/2r} \gg N^{1/2r} q^{-1/4r^2}  , \qquad \log z \ll (\log q)^{1/2}, 
$$
where all implied constants are absolute.
Therefore we now derive from~\eqref{eq:W prelim} that 
$$
\frac{W}{|\cU_z(U)|V}\le c_0N^{1-1/r}q^{(r+1)/4r^2}(\log q)^{1/4r}
$$
for some absolute constant $c_0$. 
 Substituting the above into~\eqref{eq:Win} gives
\begin{align*}
& \left|\sum_{M<n\le M+N}\chi(n)\right|\  \\ & \qquad \le  c_0N^{1-1/r}q^{(r+1)/4r^2}(\log q)^{1/4r}+\frac{c_1}{2} N^{1-1/r}q^{(r+1)/4r^2}(\log q)^{1/4r},
\end{align*}
from which~\eqref{eq:wts} follows on taking $c_1=2c_0$.

\section*{Acknowledgement}

The authors are grateful to John Friedlander and Ilya Shkredov for very 
useful discussions. The authors are also grateful to the referee for the careful reading 
of the manuscript. 

The first and second authors were supported by  the  Australian Research Council Grant DP170100786. The third author were supported by an Australian Government Research Training Program (RTP) Scholarship.


\begin{thebibliography}{99}

\bibitem{BobGold}   J. W. Bober and L. Goldmakher, {\it P{\'o}lya--Vinogradov and the least quadratic nonresidue},  
Math. Ann. {\bf 366} (2016),  853--863. 

\bibitem{Book} A. R. Booker, {\it Quadratic class numbers and character sums\/},  Math. Comp. {\bf  75} (2006), 1481--1492.

\bibitem{Bur0}
 D.A. Burgess
    {\it The distribution of quadratic residues and non-residues}
    Mathematika, {\bf 4},  (1957),  106--112.

\bibitem{Bur1}
 D. A. Burgess, {\it On character sums and L-series, I},  Proc. London Math. Soc. {\bf 12} (1962), 193--206.
 
\bibitem{Bur2}
 D. A. Burgess, {\it On character sums and L-series, II}, Proc. London Math. Soc. {\bf 13} (1963), 524--536.
 
 \bibitem{Bur3}
 D. A. Burgess, {\it A note on the distribution of quadratic residues and non-residues}, J.  London Math. Soc. {\bf 38} (1963),  253--256].
 
\bibitem{DHG} H. G. Diamond, H. Halberstam and W. F. Galway, {\it A Higher-dimensional sieve method: 
With Procedures for Computing Sieve Functions}, Cambridge Tracts in Math. {\bf 177}, Cambridge Univ. Press, 
Cambridge, 2008.

\bibitem{Ell}  P. D. T. A. Elliott, {\it Some remarks about multiplicative functions of modulus $\le 1$},
 Analytic number theory (Allerton Park, IL, 1989), Progr. Math., 85, Birkh�user Boston, Boston, MA, 1990,  159--164.


\bibitem{FrIw}
J.~B.~Friedlander and H.~Iwaniec,
{\it Estimates for character sums},
 Proc. Amer. Math. Soc., {\bf 119} (1993), 365--372.

\bibitem{FI}
J.~B.~Friedlander and H.~Iwaniec, {\it Opera de cribro}, 
Colloquium Publications~{\bf 57}  American Math. Soc., Providence, RI., 2010.
 

\bibitem{FrGo} E.  Fromm and L. Goldmakher, {\it Improving the Burgess bound via P{\'o}lya--Vinogradov}, 
Preprint, 2017 (available from \url{http://arxiv.org/abs/1706.03002}).

\bibitem{GrSo1} A. Granville and K. Soundararajan, 
`Large character sums'
{\it J. Amer. Math. Soc.\/}  {\bf 14} (2001),  365--397.

\bibitem{GrSo2}
A.~Granville and K.~Soundararajan,
{\it Decay of mean values of multiplicative functions}, 
 Canad. J. Math. {\bf 55} (2003),   1191--1230. 

\bibitem{Hild} A. Hildebrand,  {\it A note on Burgess' character sum estimate}, 
C. R. Math. Rep. Acad. Sci. Canada {\bf 8} (1986), no. 1, 35--37. 


\bibitem{GoLa}  L. Goldmakher, and Y. Lamzouri, {\it Large even order character sums}, 
Proc. Amer. Math. Soc. {\bf 119} ( 2014), 2609--2614. 


\bibitem{IwKo} H. Iwaniec and E. Kowalski, {\it Analytic Number Theory}, 
Colloquium Publications {\bf 53}  American Math. Soc., Providence, RI., 2004.

\bibitem{LaMa} Y. Lamzouri and A. P. Mangerel, {\it     Large odd order character sums and 
improvements of the  P{\'o}lya--Vinogradov inequality},  
Preprint, 2017 (available from \url{http://arxiv.org/abs/1701.01042}).

\bibitem{MV}
H. L. Montgomery and R. C. Vaughan, {\it Exponential
sums with multiplicative coefficients},
Invent. Math. {\bf 43} (1977), 69--82



\bibitem{Pol} G. P{\'o}lya, {\it Ueber die Verteilung der quadratischen Reste und Nichtreste}, 
Nachr. Akad. Wiss. Goettingen, 1918,  21--29.  


\bibitem{Tao}
T. Tao. {\it The Elliott-Halberstam conjecture implies the Vinogradov least quadratic
nonresidue conjecture},
Algebra Number Theory
{\bf 9} (2015), 1005--1034.

\bibitem{Trev}
E. Trevi{\~n}o , {\it The Burgess inequality and the least $k$-th power non-residue}, Int. J. Number Theory
{\bf 11} (2015),  1--26.

\bibitem{Vin} I. M. Vinogradov, {\it Sur la distribution des residus and nonresidus des puissances}, 
 J. Soc. Phys. Math. Univ. Permi,  1918, 18--28.  
 
\end{thebibliography}
\end{document}